\documentclass{sig-alternate}

\usepackage{amsmath,amssymb,graphicx,epsfig}
\usepackage{listings}

\usepackage[colorlinks=true,linkcolor=blue,citecolor=blue,pdfpagelabels=false]{hyperref}

  \newtheorem{theorem}{Theorem}

  \newtheorem{corollary}{Corollary}
  \newtheorem{proposition}{Proposition}
  \newdef{example}{Example}
  \newdef{remark}{Remark}

\newcommand{\qede}{\hspace*{\fill}$\Diamond$\medskip}
\newcommand{\egcaption}[1]{({#1})}

\newcommand{\op}[1]{\ensuremath{\operatorname{#1}}}
\renewcommand{\Re}{\op{Re}\,}
\renewcommand{\Im}{\op{Im}\,}
\newcommand{\e}{\mathrm{e}}

\newcommand{\md}{\mathrm{d}}
\newcommand{\id}{\,\mathrm{d}}

\newcommand{\ontop}[2]{\ensuremath{\genfrac{}{}{0pt}{}{#1}{#2}}}

\newcommand{\pFq}[5]{\ensuremath{{}_{#1}F_{#2} \left( \genfrac{}{}{0pt}{}{#3}{#4} \bigg| {#5} \right)}}
\newcommand{\Li}{\op{Li}}

\newcommand{\Gl}[2]{\op{Gl}_{#1}\left( {#2} \right)}
\newcommand{\Cl}[2]{\op{Cl}_{#1}\left( {#2} \right)}
\newcommand{\mgl}[1]{\Gl{#1}{\frac{\pi}{3}}}
\newcommand{\mcl}[1]{\Cl{#1}{\frac{\pi}{3}}}

\newcommand{\Ls}[2]{\op{Ls}_{#1}\left( {#2} \right)}
\newcommand{\LsD}[3]{\op{Ls}_{#1}^{(#2)}\left( {#3} \right)}

\newcommand{\LshD}[3]{\op{Lsh}_{#1}^{(#2)}\left( {#3} \right)}
\newcommand{\Lsc}[3]{\op{Lsc}_{#1,#2}\left( {#3} \right)}

\begin{document}

\conferenceinfo{ISSAC'11,} {June 8--11, 2011, San Jose, California, USA.} 
\CopyrightYear{2011} 
\crdata{978-1-4503-0675-1/11/06} 
\clubpenalty=10000 
\widowpenalty = 10000

\title{Special Values of Generalized Log-sine Integrals}

\numberofauthors{2}

\author{
\alignauthor
Jonathan M. Borwein\\
       % \affaddr{Centre for Computer-assisted Research Mathematics and its Applications (CARMA)}\\
       % \affaddr{School of Mathematical and Physical Sciences}\\
       \affaddr{University of Newcastle}\\
       \affaddr{Callaghan, NSW 2308, Australia}\\
       \email{jonathan.borwein@newcastle.edu.au}
\alignauthor
Armin Straub\\
       \affaddr{Tulane University}\\
       \affaddr{New Orleans, LA 70118, USA}\\
       \email{astraub@tulane.edu}
}

\date{\today}

\maketitle

\begin{abstract}
  We study generalized log-sine integrals at special values. At $\pi$ and
  multiples thereof explicit evaluations are obtained in terms of Nielsen
  polylogarithms at $\pm1$. For general arguments we present algorithmic
  evaluations involving Nielsen polylogarithms at related arguments. In
  particular, we consider log-sine integrals at $\pi/3$ which evaluate in terms
  of polylogarithms at the sixth root of unity. An implementation of our
  results for the computer algebra systems \emph{Mathematica} and SAGE is
  provided.
\end{abstract}

% \category{H.4}{Information Systems Applications}{Miscellaneous}
% %A category including the fourth, optional field follows...
% \category{D.2.8}{Software Engineering}{Metrics}[complexity measures, performance measures]



\section{Introduction}

For $n=1,2, \ldots$ and $k\ge0$, we consider the (generalized) \emph{log-sine
integrals} defined by
\begin{equation}\label{eq:slxm}
  \LsD{n}{k}{\sigma}
  := - \int_{0}^{\sigma}\theta^k\,\log^{n-1-k} \left|2\,\sin \frac \theta 2\right| \,{\md\theta}.
\end{equation}
The modulus is not needed for $0 \le \sigma \le 2\pi$. For $k=0$ these are the
(basic) log-sine integrals $\Ls{n}{\sigma}:=\LsD{n}{0}{\sigma}$.  Various
log-sine integral evaluations may be found in \cite[\S7.6 \& \S7.9]{lewin2}.

In this paper, we will be concerned with evaluations of the log-sine integrals
$\LsD{n}{k}{\sigma}$ for special values of $\sigma$.  Such evaluations are
useful for physics \cite{lsjk}: log-sine integrals appeared for instance in
recent work on the $\varepsilon$-expansion of various Feynman diagrams in the
calculation of higher terms in the $\varepsilon$-expansion, \cite{dk0-eps, kv0-bin,
dk1-eps, davydychev-eps, kalmykov-eps}. Of particular importance are the
log-sine integrals at the special values $\pi/3$, $\pi/2$, $2\pi/3$, $\pi$. The
log-sine integrals also appear in many settings in number theory and analysis:
classes of inverse binomial sums can be expressed in terms of generalized
log-sine integrals, \cite{davydychev-bin, mcv}.

In Section \ref{sec:atpi} we focus on evaluations of log-sine and related
integrals at $\pi$. General arguments are considered in Section \ref{sec:lsx}
with a focus on the case $\pi/3$ in Section \ref{sec:lspi3}.  Imaginary
arguments are briefly discussed in \ref{sec:lsimag}.  The  results obtained are
suitable for implementation in a computer algebra system. Such an
implementation is provided for \emph{Mathematica} and SAGE, and is described in
Section \ref{sec:program}. This complements existing packages such as
\texttt{lsjk} \cite{lsjk} for numerical evaluations of log-sine integrals or
\texttt{HPL} \cite{hpl} as well as \cite{polylog-num} for working with multiple
polylogarithms.

Further motivation for such evaluations was sparked by our recent study
\cite{logsin1} of certain \emph{multiple Mahler measures}.  For $k$ functions
(typically Laurent polynomials) in $n$ variables the multiple Mahler measure
$\mu(P_1,P_2, \ldots, P_k)$, introduced in \cite{klo}, is defined by
\begin{equation*}
  \int_0^1 \cdots \int_0^1 \prod_{j=1}^k \log \left| P_j\left(e^{2\pi i t_1},
    \ldots, e^{2\pi i t_n}\right)\right| \md t_1 \md t_2 \ldots \md t_n.
\end{equation*}
When $P=P_1=P_2= \cdots =P_k$ this devolves to a \emph{higher Mahler measure},
$\mu_k(P)$, as introduced and examined in \cite{klo}. When $k=1$ both reduce
to the standard (logarithmic) \emph{Mahler measure} \cite{boyd}.

The multiple Mahler measure
\begin{equation}\label{eq:defmu1xys}
  \mu_k(1+x+y_*) := \mu(1+x+y_1, 1+x+y_2, \ldots, 1+x+y_k)
\end{equation}
was studied by Sasaki \cite[Lemma 1]{sasaki} who provided an evaluation of
$\mu_2(1+x+y_*)$. It was observed in \cite{logsin1} that
\begin{equation}\label{eq:mukb}
  \mu_k(1+x+y_*) = \frac1\pi\Ls{k+1}{\frac\pi3} - \frac1\pi\Ls{k+1}{\pi}.
\end{equation}
Many other Mahler measures studied in \cite{logsin1,logsin2} were shown to have
evaluations involving generalized log-sine integrals at $\pi$ and $\pi/3$ as
well.

To our knowledge, this is the most exacting such study undertaken --- perhaps
because it would be quite impossible without modern computational tools and
absent a use of  the quite recent understanding of multiple polylogarithms and
multiple zeta values \cite{b3l}.

\section{Preliminaries}

In the following, we will denote the \emph{multiple polylogarithm} as studied
for instance in \cite{mcv} and \cite[Ch. 3]{bbg} by
\begin{equation*}
  \Li_{a_1,\ldots,a_k}(z)
  := \sum_{n_1>\cdots>n_k>0} \frac{z^{n_1}}{n_1^{a_1} \cdots n_k^{a_k}}.
\end{equation*}
For our purposes, the $a_1,\ldots,a_k$ will usually be positive integers and
$a_1\ge2$ so that the sum converges for all $|z|\le1$.  For example,
$\Li_{2,1}(z) =\sum_{k=1}^\infty \frac{z^k}{k^2}\sum_{j=1}^{k-1} \frac1j$.  In
particular, $\Li_k(x) := \sum_{n=1}^\infty \frac{x^n}{n^k}$ is the
\emph{polylogarithm of order $k$}. The usual notation will be used for
repetitions so that, for instance, $\Li_{2,\{1\}^3}(z)=\Li_{2,1,1,1}(z)$.

Moreover, \emph{multiple zeta values} are denoted by
\begin{equation*}
  \zeta(a_1,\ldots,a_k) := \Li_{a_1,\ldots,a_k}(1).
\end{equation*}
Similarly, we consider the \emph{multiple Clausen functions} ($\op{Cl}$) and
\emph{multiple Glaisher functions} ($\op{Gl}$) of depth $k$ and weight
$w=a_1+\ldots+a_k$ defined as
\begin{align}
  \Cl{a_1,\ldots,a_k}{\theta} &= \left\{ \begin{array}{ll}
    \Im \Li_{a_1,\ldots,a_k}(e^{i\theta}) & \text{if $w$ even}\\
    \Re \Li_{a_1,\ldots,a_k}(e^{i\theta}) & \text{if $w$ odd}
  \end{array} \right\}, \label{eq:defcl}\\
  \Gl{a_1,\ldots,a_k}{\theta} &= \left\{ \begin{array}{ll}
    \Re \Li_{a_1,\ldots,a_k}(e^{i\theta}) & \text{if $w$ even}\\
    \Im \Li_{a_1,\ldots,a_k}(e^{i\theta}) & \text{if $w$ odd}
  \end{array} \right\}, \label{eq:defgl}
\end{align}
in accordance with \cite{lewin2}.
% Thus
% \begin{equation}\label{eq:defcl2}
%   \Ls{2}{\theta} = \Cl{2}{\theta}
%   = \sum_{n=1}^\infty \frac{\sin(n\theta)}{n^2}.
% \end{equation}
% As illustrated by \eqref{eq:defcl2} and later in \eqref{eq:mgl41}, the Clausen
% and Glaisher functions alternate between being cosine and sine series with the
% parity of the dimension.
Of particular importance will be the case of $\theta
= \pi/3$ which has also been considered in \cite{mcv}.

Our other notation and usage is largely consistent with that in \cite{lewin2}
and that in the newly published \cite{NIST} in which most of the requisite
material is described.  Finally, a recent elaboration of what is meant when we
speak about evaluations and ``closed forms'' is to be found in \cite{closed}.

\section[Evaluations at pi]{Evaluations at $\pi$}
\label{sec:atpi}

\subsection[Basic log-sine integrals at pi]{Basic log-sine integrals at $\pi$}
\label{sec:lspi}

The exponential generating function, \cite{lewin,lewin2},
\begin{equation}\label{eq:lsbpiegf}
  -\frac1\pi \sum _{m=0}^{\infty }\Ls{m+1}{\pi} \frac {\lambda^m}{m!}
  = \frac{\Gamma\left(1+\lambda\right)}{\Gamma^2\left(1+\frac\lambda2\right)}
  = \binom{\lambda}{\frac\lambda2}
\end{equation}
is well-known and implies the recurrence
\begin{align}\label{eq:lsatpirec}
  \frac{(-1)^n}{n!} &\Ls{n+2}{\pi} = \pi\, \alpha(n+1) \nonumber\\
  &+ \sum _{k=1}^{n-2} \frac{(-1)^k}{(k+1)!} \,\alpha  (n-k) \Ls{k+2}{\pi},
\end{align}
where $\alpha(m) = (1-2^{1-m}) \zeta(m)$.

\begin{example}\label{ex:pi}
  \egcaption{Values of $\Ls{n}{\pi}$}
  We have $\Ls{2}{\pi}=0$ and
  \begin{align*}\label{eq:pi}
    -\Ls{3}{\pi} &= {\frac {1}{12}}\,\pi^3\\
    \Ls{4}{\pi} &= \frac32\pi\,\zeta(3)\\
    -\Ls{5}{\pi} &= \frac{19}{240}\,\pi^5\\
    \Ls{6}{\pi} &= {\frac{45}{2}}\,\pi \,\zeta(5)+\frac54\,\pi^3\zeta(3)\\
    -\Ls{7}{\pi} &= {\frac {275}{1344}}\,\pi^7+{\frac{45}{2}}\,\pi \, \zeta(3)^2\\
    \Ls{8}{\pi} &= {\frac {2835}{4}}\,\pi \,\zeta(7)+ {\frac{315}{8}}\,\pi^3\zeta(5)
      +{\frac{133}{32}}\,\pi^5\zeta(3),
  \end{align*}
  and so forth. The fact that each integral is a multivariable rational
  polynomial in $\pi$ and zeta values follows directly from the recursion
  \eqref{eq:lsatpirec}. Alternatively, these values may be conveniently obtained
  from \eqref{eq:lsbpiegf} by a computer algebra system.
  For instance, in \emph{Mathematica} the code\\
  % \texttt{simplify(subs(x=0,diff(Pi*binomial(x,x/2),x,6)))}\\
  \texttt{FullSimplify[D[-Binomial[x,x/2], \{x,6\}] /.x->0]}\\
  produces the above evaluation of $\Ls{6}{\pi}$.
  % \texttt{for k to 6 do simplify(subs(x=0,diff(Pi*binomial(x,x/2),x\$k))) od}
  % \texttt{Table[D[Binomial[x,x/2], {x,k}] /. x->0, {k,1,6}] // FullSimplify}
  \qede
\end{example}

\subsection{The log-sine-cosine integrals}

The log-sine-cosine integrals
\begin{equation}\label{eq:deflsc}
  \Lsc{m}{n}{\sigma} := -\int_{0}^{\sigma}\log^{m-1} \left|2\,\sin \frac \theta 2\right|\,\log^{n-1} \left|2\,\cos \frac \theta 2\right| \id\theta
\end{equation}
appear in physical applications as well, see for instance \cite{dk1-eps,kalmykov-eps}.
They have also been considered by Lewin, \cite{lewin,lewin2}, and he
demonstrates how their values at $\sigma=\pi$ may be obtained much the same as
those of the log-sine integrals in Section \ref{sec:lspi}. As observed in
\cite{logsin2}, Lewin's result can be put in the form
\begin{align}\label{eq:egflscpi}
  -\frac1\pi \sum _{m,n=0}^\infty & \Lsc{m+1}{n+1}{\pi} \frac{x^m}{m!}\frac{y^n}{n!}
  = \frac{2^{x+y}}{\pi} \frac{\Gamma\left( \frac{1+x}{2} \right) \Gamma\left( \frac{1+y}{2} \right)}
  {\Gamma\left( 1+\frac{x+y}2 \right)} \nonumber\\
  &= \binom{x}{x/2}\binom{y}{y/2} \frac{\Gamma\left( 1+\frac{x}{2} \right)
  \Gamma\left( 1+\frac{y}{2} \right)}{\Gamma\left( 1+\frac{x+y}{2} \right)}.
\end{align}
The last form makes it clear that this is an extension of
\eqref{eq:lsbpiegf}.

The notation $\op{Lsc}$ has been introduced in \cite{dk1-eps} where evaluations
for other values of $\sigma$ and low weight can be found.

\subsection[Log-sine integrals at pi]{Log-sine integrals at $\pi$}
\label{sec:lsxpi}

As Lewin \cite[\S7.9]{lewin2} sketches, at least for small values of $n$ and
$k$, the generalized log-sine integrals $\LsD{n}{k}{\pi}$ have closed forms
involving zeta values and Kummer-type constants such as $\Li_4(1/2)$. This
will be made more precise in Remark~\ref{rk:lsxpiLi}. Our analysis starts with
the generating function identity
\begin{align}\label{eq:lspiegf}
 &-\sum_{n,k \ge 0} \LsD{n+k+1}{k}{\pi} \frac{\lambda^n}{n!}\frac{(i\mu)^k}{k!}
 = \int_0^\pi \left( 2\sin\frac\theta2 \right)^\lambda \e^{i\mu\theta} \id\theta \nonumber\\
 &= i\e^{i\pi \frac\lambda2}\,B_1\left(\mu-\tfrac\lambda2,1+\lambda\right)- i{\e^{i\pi \mu}}
 B_{1/2}\left(\mu-\tfrac\lambda2,-\mu-\tfrac\lambda2\right)
\end{align}
given in \cite{lewin2}. Here $B_x$ is the \emph{incomplete Beta} function:
\begin{equation*}
  B_x(a,b) = \int_0^x t^{a-1} (1-t)^{b-1} \id t.
\end{equation*}
We shall show that with care --- because of the singularities at zero ---
\eqref{eq:lspiegf} can be differentiated as needed as suggested by Lewin.

Using the identities, valid for $a,b>0$ and $0<x<1$,
\begin{align*}
  B_x(a,b)
  &= \frac {x^a (1-x)^{b-1}}{a}\pFq21{1-b,1}{a+1}{\frac x{x-1}} \\
  &=\frac {x^a (1-x)^b}{a}\pFq21{a+b,1}{a+1}{x},
  % = \frac {x^a}{a}\pFq21{a,1-b}{a+1}{x}
\end{align*}
found for instance in \cite[\S8.17(ii)]{NIST},
the generating function \eqref{eq:lspiegf} can be rewritten as
\begin{align*}
    i \e^{i\pi\frac\lambda2} \left( B_1\left(\mu-\frac\lambda2,1+\lambda\right)
    - B_{-1}\left(\mu-\frac\lambda2,1+\lambda\right) \right).
\end{align*}
Upon expanding this  we obtain the following
computationally more accessible generating function for $\LsD{n+k+1}{k}{\pi}$:

\begin{theorem}\label{thm:gfb}
  For $2|\mu| <\lambda <1$ we have
  \begin{align}\label{eq:gfsum}
    -\sum_{n,k \ge 0}&\LsD{n+k+1}{k}{\pi}\frac{\lambda^n}{n!}\frac{(i\mu)^k}{k!} \nonumber\\
    &= i \sum_{n\ge0} \binom{\lambda}{n}
    \frac{(-1)^n \e^{i\pi\frac\lambda2} - \e^{i\pi\mu}}{\mu-\frac\lambda2+n}.
  \end{align}
\end{theorem}

We now show how the log-sine integrals $\LsD{n}{k}{\pi}$ can quite
comfortably be extracted from \eqref{eq:gfsum} by differentiating
its right-hand side. The case $n=0$ is covered by:

\begin{proposition}\label{prop:diffn0}
  We have
  \begin{align*}
    \frac{\md^k}{\md\mu^k} \frac{\md^m}{\md\lambda^m} i
    \frac{\e^{i\pi\frac\lambda2} - \e^{i\pi\mu}}{\mu-\frac\lambda2}\bigg|_{\ontop{\lambda=0}{\mu=0}}
    = \frac{\pi}{2^m} (i \pi)^{m+k} B(m+1,k+1).
  \end{align*}
\end{proposition}

\begin{proof}
  This may be deduced from
  \begin{align*}
    \frac{\e^x - \e^y}{x-y}
    &= \sum_{m,k\ge0} \frac{x^m y^k}{(k+m+1)!} \\
    &= \sum_{m,k\ge0} B(m+1,k+1) \frac{x^m}{m!} \frac{y^k}{k!}
  \end{align*}
  upon setting $x=i\pi\lambda/2$ and $y=i\pi\mu$.
\end{proof}

The next proposition is most helpful in differentiation of the right-hand side
of \eqref{eq:gfsum} for $n\ge1$, Here, we denote a \emph{multiple harmonic
number} by
\begin{equation}\label{eq:defharmonic}
  H_{n-1}^{[\alpha]}
  := \sum_{n>i_1>i_2>\ldots>i_{\alpha}}
  \frac{1}{i_1 i_2 \cdots i_{\alpha}}.
\end{equation}
If $\alpha=0$ we set $H_{n-1}^{[0]}:=1$.

\begin{proposition}\label{prop:diffbinom}
  For $n\ge1$
  \begin{equation}\label{eq:diffbinom}
   \frac{(-1)^\alpha}{\alpha!} \left( \frac{\md}{\md\lambda}\right)^{\alpha}
   \binom{\lambda}{n} \bigg|_{\lambda=0}
   = \frac{(-1)^{n}}{n} \,H_{n-1}^{[\alpha-1]}.
  \end{equation}
\end{proposition}

Note that, for $\alpha\ge0$,
\begin{align*}
  \sum_{n\ge0} \frac{(\pm1)^n}{n^\beta} H_{n-1}^{[\alpha]}
  = \Li_{\beta,\{1\}^\alpha}(\pm1)
\end{align*}
which shows that the evaluation of the log-sine integrals will involve Nielsen
polylogarithms at $\pm1$, that is polylogarithms of the type $\Li_{a,\{1\}^b}(\pm1)$.

Using the Leibniz rule coupled with Proposition \ref{prop:diffbinom} to
differentiate \eqref{eq:gfsum} for $n\ge1$ and Proposition \ref{prop:diffn0} in
the case $n=0$, it is possible to explicitly write $\LsD{n}{k}{\pi}$ as a
finite sum of Nielsen polylogarithms with coefficients only
being rational multiples of powers of $\pi$. The process is now exemplified
for $\LsD{4}{2}{\pi}$ and $\LsD{5}{1}{\pi}$.

\begin{example}\label{eg:ls42}
  \egcaption{$\LsD{4}{2}{\pi}$}
  To find $\LsD{4}{2}{\pi}$ we differentiate \eqref{eq:gfsum} once
  with respect to $\lambda$ and twice with respect to $\mu$. To
  simplify computation, we exploit the fact that the result will be
  real which allows us to neglect imaginary parts:
  \begin{align*}
    -\LsD{4}{2}{\pi} &= \frac{\md^2}{\md\mu^2} \frac{\md}{\md\lambda} i \sum_{n\ge0} \binom{\lambda}{n}
    \frac{(-1)^n \e^{i\pi\frac\lambda2} - \e^{i\pi\mu}}{\mu-\frac\lambda2+n}\bigg|_{\lambda=\mu=0} \\
    &= 2\pi \sum_{n\ge1} \frac{(-1)^{n+1}}{n^3}
    = \frac32\,\pi\zeta(3).
  \end{align*}
  In the second step we were able to drop the term corresponding to $n=0$
  because its contribution $-i\pi^4/24$ is purely imaginary as follows a priori
  from Proposition \ref{prop:diffbinom}.
  \qede
\end{example}

\begin{example}\label{eg:ls51}
  \egcaption{$\LsD{5}{1}{\pi}$}
  Similarly, setting
  \[ \Li^\pm_{a_1,\ldots,a_n} := \Li_{a_1,\ldots,a_n}(1) - \Li_{a_1,\ldots,a_n}(-1) \]
  we obtain $\LsD{5}{1}{\pi}$ as
  \begin{align*}
    -\LsD{5}{1}{\pi}
    ={}& \frac34 \sum_{n\ge1} \frac{8(1-(-1)^n)}{n^4}
    \left( n H_{n-1}^{[2]} - H_{n-1} \right) \\
    &+ \frac{6(1-(-1)^n)}{n^5} - \frac{\pi^2}{n^3} \\
    ={}& 6\Li^\pm_{3,1,1} - 6\Li^\pm_{4,1}
    + \frac92 \Li^\pm_5 - \frac34 \pi^2 \zeta(3) \\
    % ={}& 2\,\lambda_5\left(\tfrac12\right) - \frac34\pi^2\zeta(3) - \frac{93}{32}\zeta(5).
    ={}& -6\Li_{3,1,1}(-1) + \frac{105}{32}\zeta(5) - \frac14\pi^2\zeta(3).
  \end{align*}
  The last form is what is automatically produced by our program, see Example
  \ref{eg:mathematica}, and is obtained from the previous expression by
  reducing the polylogarithms as discussed in Section \ref{sec:reducing}.
  \qede
\end{example}

The next example hints at the rapidly growing complexity of these integrals,
especially when compared to the evaluations given in Examples \ref{eg:ls42} and
\ref{eg:ls51}.

\begin{example}\label{eg:ls61}
  \egcaption{$\LsD{6}{1}{\pi}$}
  Proceeding as before we find
  \begin{align}
    -\LsD{6}{1}{\pi}
    ={}& -24\Li^\pm_{3,1,1,1} + 24\Li^\pm_{4,1,1} - 18\Li^\pm_{5,1} + 12\Li^\pm_{6} \nonumber\\
    &+ 3\pi^2\zeta(3,1) - 3\pi^2\zeta(4) + \frac{\pi^6}{480} \nonumber\\
    ={}& 24\Li_{3,1,1,1}(-1) - 18\Li_{5,1}(-1) \nonumber\\
    &+ 3\zeta(3)^2 - \frac{3}{1120}\pi^6. \label{eq:ls61}
  \end{align}
  In the first equality, the term $\pi^6/480$ is the one corresponding to $n=0$
  in \eqref{eq:gfsum} obtained from Proposition \ref{prop:diffn0}. The
  second form is again the automatically reduced output of our program.
  \qede
\end{example}

\begin{remark}\label{rk:lsxpiLi}
  From the form of \eqref{eq:gfsum} and \eqref{eq:diffbinom} we find that the
  log-sine integrals $\LsD{n}{k}{\pi}$ can be expressed in terms of $\pi$ and
  Nielsen polylogarithms at $\pm1$.  Using the duality results in
  \cite[\S 6.3, and Example 2.4]{b3l} the polylogarithms at $-1$ may be
  explicitly reexpressed as multiple polylogarithms at $1/2$. Some examples are
  given in \cite{logsin1}.

  Particular cases of Theorem \ref{thm:gfb} have been considered in \cite{lsjk}
  where explicit formulae are given for $\LsD{n}{k}{\pi}$ where $k=0,1,2$.
  \qede
\end{remark}

\subsection[Log-sine integrals at 2pi]{Log-sine integrals at $2\pi$}
\label{sec:ls2pi}

As observed by Lewin \cite[7.9.8]{lewin2}, log-sine integrals at $2\pi$ are
expressible in terms of zeta values only. If we proceed as in the case of
evaluations at $\pi$ in \eqref{eq:lspiegf} we find that the resulting integral
now becomes expressible in terms of gamma functions:
\begin{align}\label{eq:ls2piegf}
 -\sum_{n,k \ge 0} \LsD{n+k+1}{k}{2\pi} \frac{\lambda^n}{n!}\frac{(i\mu)^k}{k!}
 &= \int_0^{2\pi} \left( 2\sin\frac\theta2 \right)^\lambda \e^{i\mu\theta} \id\theta \nonumber\\
 &= 2\pi \e^{i\mu\pi} \binom{\lambda}{\frac\lambda2 + \mu}
\end{align}
The special case $\mu=0$, in the light of \eqref{eq:lsb2pi} which gives
$\Ls{n}{2\pi}=2\Ls{n}{\pi}$, recovers \eqref{eq:lsbpiegf}.

We may now extract log-sine integrals $\LsD{n}{k}{2\pi}$ in a similar way as
described in Section \ref{sec:lspi}.

\begin{example}\label{eg:ls52at2pi}
  For instance,
  \begin{align*}
    \LsD{5}{2}{2\pi}
    &= -\frac{13}{45} \pi^5.
  \end{align*}
  We remark that this evaluation is incorrectly given in \cite[(7.144)]{lewin2}
  as $7\pi^5/30$ underscoring an advantage of automated evaluations over
  tables (indeed, there are more misprints in \cite{lewin2} pointed out
  for instance in \cite{dk1-eps, lsjk}).
  \qede
\end{example}

\subsection{Log-sine-polylog integrals}

Motivated by the integrals $\op{LsLsc}_{k,i,j}$ defined in \cite{kalmykov-eps}
we show that the considerations of Section \ref{sec:lsxpi} can be extended to
more involved integrals including
\begin{align*}
  \LsD{n}{k}{\pi;d}
  &:=-\int_0^\pi \theta^k\log^{n-k-1}\left(2\sin\frac\theta 2\right)\Li_d(\e^{i\theta}) \id\theta.
\end{align*}
On expressing $\Li_d(\e^{i\theta})$ as a series, rearranging, and applying
Theorem \ref{thm:gfb}, we obtain the following exponential generating function
for $\LsD{n}{k}{\pi;d}$:

\begin{corollary}\label{ex:clls}
  For $d \ge 0$ we have
  \begin{align}\label{eq:gflsum}
    -\sum_{n,k \ge 0} &\LsD{n+k+1}{k}{\pi;d}\frac{\lambda^n}{n!}\frac{(i\mu)^k}{k!} \nonumber\\
    &= i \sum_{n\ge1} H_{n,d}(\lambda) \frac{\e^{i\pi\frac\lambda2}
    - (-1)^n \e^{i\pi\mu}}{\mu-\frac\lambda2+n}
  \end{align}
  where
  \begin{align}\label{def:hd}
    H_{n,d}(\lambda)
    &:= \sum_{k=0}^{n-1} \frac{(-1)^k \binom{\lambda}{k}}{(n-k)^d}.
  \end{align}
\end{corollary}

We note for $0 \le \theta \le \pi$ that
$\Li_{-1}(\e^{i\theta})=-1/\left(2\sin \frac \theta2\right)^2$,
$\Li_0(\e^{i\theta})=-\frac12+\frac i2\cot \frac \theta2$, while
$\Li_1(\e^{i\theta}) =-\log \left (2\sin \frac \theta2\right
)+i\frac{\pi-\theta}2$, and $\Li_2(\e^{i\theta})
=\zeta(2)+\frac\theta2\left(\frac\theta2-\pi\right)+i\Cl{2}{\theta}$.

\begin{remark}
  Corresponding results for an arbitrary Dirichlet series
  $\op{L}_{\textbf{a},d}(x) := \sum_{n \ge 1} a_n x^n /n^{d}$ can
  be easily derived in the same fashion.
  Indeed, for
  \begin{align*}%\label{def:lsLa}
    \LsD{n}{k}{\pi;\textbf{a},d}
    &:= -\int_0^\pi
    \theta^k\log^{n-k-1}\left(2\sin\frac\theta 2\right) \op{L}_{\textbf{a},d}(\e^{i\theta}) \id\theta
  \end{align*}
  one derives the exponential generating function \eqref{eq:gflsum} with
  $H_{n,d}(\lambda)$ replaced by
  \begin{align}
    H_{n,\textbf{a},d}(\lambda)
    &:= \sum_{k=0}^{n-1} \frac{(-1)^k \binom{\lambda}{k}a_{n-k}}{(n-k)^d}.
  \end{align}
  This allows for $\LsD{n}{k}{\pi;\textbf{a},d}$ to be extracted for many
  number theoretic functions. It does not however seem to cover any of the
  values of the ${\rm LsLsc}_{k,i,j}$ function defined in \cite{kalmykov-eps}
  that are not already covered by Corollary \ref{ex:clls}.
  \qede
\end{remark}

\section{Quasiperiodic properties}
\label{sec:lsperiodic}

As shown in \cite[(7.1.24)]{lewin2}, it follows from the periodicity of the
integrand that, for integers $m$,
\begin{align}\label{eq:lsperiod}
  \LsD{n}{k}{2m\pi}& - \LsD{n}{k}{2m\pi-\sigma} \nonumber\\
  &= \sum_{j=0}^k (-1)^{k-j} (2m\pi)^j \binom{k}{j} \LsD{n-j}{k-j}{\sigma}.
\end{align}
Based on this quasiperiodic property of the log-sine integrals, the results of
Section \ref{sec:ls2pi} easily generalize to show that log-sine integrals at
multiples of $2\pi$ evaluate in terms of zeta values.  This is shown in Section
\ref{sec:lsmpi}.  It then follows from \eqref{eq:lsperiod} that log-sine
integrals at general arguments can be reduced to log-sine integrals at
arguments $0\le\sigma\le\pi$. This is discussed briefly in Section
\ref{sec:lsreduce}.

\begin{example}
  In the case $k=0$, we have that
  \begin{equation}\label{eq:lsb2pi}
    \Ls{n}{2m\pi} = 2m \Ls{n}{\pi}.
  \end{equation}
  For $k=1$, specializing \eqref{eq:lsperiod} to $\sigma=2m\pi$ then yields
  \[ \LsD{n}{1}{2m\pi} = 2m^2\pi \Ls{n-1}{\pi} \]
  as is given in \cite[(7.1.23)]{lewin2}.
  \qede
\end{example}

\subsection[Log-sine integrals at multiples of 2pi]{Log-sine integrals at multiples of $2\pi$}
\label{sec:lsmpi}

For odd $k$, specializing \eqref{eq:lsperiod} to $\sigma=2m\pi$, we find
\begin{align*}
  2\LsD{n}{k}{2m\pi}
  &= \sum_{j=1}^k (-1)^{j-1} (2m\pi)^j \binom{k}{j} \LsD{n-j}{k-j}{2m\pi}
\end{align*}
giving $\LsD{n}{k}{2m\pi}$ in terms of lower order log-sine integrals.

More generally, on setting $\sigma=2\pi$ in \eqref{eq:lsperiod} and summing the
resulting equations for increasing $m$ in a telescoping fashion, we arrive at
the following reduction.
We will use the standard notation
\begin{equation*}
  H_n^{(a)} := \sum_{k=1}^n k^{-a}
\end{equation*}
for \emph{generalized harmonic sums}.

\begin{theorem}
  For integers $m\ge0$,
  \begin{align*}
    \LsD{n}{k}{2m\pi}
    = \sum_{j=0}^k (-1)^{k-j} (2\pi)^j \binom{k}{j} H_m^{(-j)} \LsD{n-j}{k-j}{2\pi}.
  \end{align*}
\end{theorem}

Summarizing, we have thus shown that the generalized log-sine integrals at
multiples of $2\pi$ may always be evaluated in terms of integrals at $2\pi$. In
particular, $\LsD{n}{k}{2m\pi}$ can always be evaluated in terms of zeta values
by the methods of Section \ref{sec:ls2pi}.

\subsection{Reduction of arguments}
\label{sec:lsreduce}

A general (real) argument $\sigma$ can be written uniquely as
$\sigma=2m\pi\pm\sigma_0$ where $m\ge0$ is an integer and $0\le\sigma_0\le\pi$.
It then follows from \eqref{eq:lsperiod} and
\begin{equation*}
  \LsD{n}{k}{-\theta} = (-1)^{k+1} \LsD{n}{k}{\theta}
\end{equation*}
that $\LsD{n}{k}{\sigma}$ equals
\begin{align}
   % \LsD{n}{k}{\sigma} =
   \LsD{n}{k}{2m\pi}& \pm
  \sum_{j=0}^k (\pm1)^{k-j} (2m\pi)^j \binom{k}{j} \LsD{n-j}{k-j}{\sigma_0}.
\end{align}
Since the evaluation of log-sine integrals at multiples of $2\pi$ was
explicitly treated in Section \ref{sec:lsmpi} this implies that the
evaluation of log-sine integrals at general arguments $\sigma$ reduces
to the case of arguments $0\le\sigma\le\pi$.

\section{Evaluations at other values}
\label{sec:lsx}

In this section we first discuss a method for evaluating the generalized
log-sine integrals at arbitrary arguments in terms of Nielsen polylogarithms at
related arguments. The gist of our technique originates with Fuchs
(\cite{fuchs}, \cite[\S7.10]{lewin2}). Related evaluations appear in
\cite{dk0-eps} for $\Ls{3}{\tau}$ to $\Ls{6}{\tau}$ as well as in
\cite{dk1-eps} for $\Ls{n}{\tau}$ and $\LsD{n}{1}{\tau}$.

We then specialize to evaluations at $\pi/3$ in Section \ref{sec:lspi3}. The
polylogarithms arising in this case have been studied under the name
of \emph{multiple Clausen and Glaisher values} in \cite{mcv}.  In fact, the next
result \eqref{eq:lsb} with $\tau=\pi/3$ is a modified version of \cite[Lemma
3.2]{mcv}. We employ the notation
\begin{equation*}
  \binom{n}{a_1,\ldots,a_k} := \frac{n!}{a_1! \cdots a_k! (n-a_1-\ldots-a_k)!}
\end{equation*}
for multinomial coefficients.

\begin{theorem}\label{thm:ls2li}
  For $0\le\tau\le2\pi$, and nonnegative integers $n$, $k$ such that $n-k\ge2$,
  \begin{align}\label{eq:lsb}
    \zeta(&n-k,\{1\}^k) - \sum_{j=0}^{k} \frac{(-i\tau)^j}{j!}
    \Li_{2+k-j,\{1\}^{n-k-2}}(\e^{i\tau}) \nonumber\\
    &= \frac{i^{k+1}(-1)^{n-1}}{(n-1)!}
    \sum_{r=0}^{n-k-1} \sum_{m=0}^r \binom{n-1}{k,m,r-m} \nonumber\\
    &\quad\times\left( \frac{i}{2} \right)^r (-\pi)^{r-m} \LsD{n-(r-m)}{k+m}{\tau}.
  \end{align}
\end{theorem}

\begin{proof}
  Starting with
  \begin{equation*}
    \Li_{k,\{1\}^n}(\alpha) - \Li_{k,\{1\}^n}(1)
    = \int_1^\alpha \frac{\Li_{k-1,\{1\}^n}(z)}{z} \id z
  \end{equation*}
  and integrating by parts repeatedly, we obtain
  \begin{align}
    \sum_{j=0}^{k-2}& \frac{(-1)^j}{j!} \log^j(\alpha) \Li_{k-j,\{1\}^n}(\alpha)
    - \Li_{k,\{1\}^n}(1) \nonumber\\
    &= \frac{(-1)^{k-2}}{(k-2)!} \int_1^\alpha \frac{\log^{k-2}(z)
    \Li_{\{1\}^{n+1}}(z)}{z} \id z. \label{eq:liparts}
  \end{align}
  Letting $\alpha=\e^{i\tau}$ and changing variables to $z=\e^{i\theta}$, as
  well as using
  \begin{equation*}
    \Li_{\{1\}^n}(z) = \frac{(-\log(1-z))^n}{n!},
  \end{equation*}
  the right-hand side of \eqref{eq:liparts} can be rewritten as
  \begin{equation*}
    \frac{(-1)^{k-2}}{(k-2)!} \frac{i}{(n+1)!}
    \int_0^\tau (i\theta)^{k-2} \left( - \log\left(1-\e^{i\theta}\right) \right)^{n+1} \id\theta.
  \end{equation*}
  Since, for $0\le\theta\le2\pi$ and the principal branch of the logarithm,
  \begin{equation}\label{eq:log2logsin}
    \log(1-\e^{i\theta}) = \log\left| 2\sin\frac\theta2 \right| + \frac{i}{2}(\theta-\pi),
  \end{equation}
  this last integral can now be expanded in terms of generalized log-sine
  integrals at $\tau$.
  \begin{align}
    \zeta(k,\{1\}^n) &- \sum_{j=0}^{k-2} \frac{(-i\tau)^j}{j!}
    \Li_{k-j,\{1\}^n}(\e^{i\tau}) \nonumber\\
    ={}& \frac{(-i)^{k-1}}{(k-2)!}\frac{(-1)^n}{(n+1)!}
    \sum_{r=0}^{n+1} \sum_{m=0}^r \binom{n+1}{r} \binom{r}{m} \nonumber\\
    &\quad\quad\left( \frac{i}{2} \right)^r (-\pi)^{r-m} \LsD{n+k-(r-m)}{k+m-2}{\tau}.
  \end{align}
  Applying the \emph{MZV duality formula} \cite{b3l}, we have
  \begin{equation*}
    \zeta(k,\{1\}^n) = \zeta(n+2,\{1\}^{k-2}),
  \end{equation*}
  and a change of variables yields the claim.
\end{proof}

We recall that the real and imaginary parts of the multiple polylogarithms are
Clausen and Glaisher functions as defined in \eqref{eq:defcl} and
\eqref{eq:defgl}.

\begin{example}\label{eg:ls41tau}
  Applying \eqref{eq:lsb} with $n=4$ and $k=1$ and solving for
  $\LsD{4}{1}{\tau}$ yields
  \begin{align*}
    \LsD{4}{1}{\tau}
    &= 2\zeta(3,1) - 2\Gl{3,1}{\tau} - 2\tau\Gl{2,1}{\tau} \\
    & + \frac14\LsD{4}{3}{\tau} - \frac12\pi\LsD{3}{2}{\tau} + \frac14\pi^2\LsD{2}{1}{\tau} \\
    &= \frac{1}{180}\pi^4 - 2\Gl{3,1}{\tau} - 2\tau\Gl{2,1}{\tau} \\
    & - \frac{1}{16}\tau^4 + \frac16\pi\tau^3 - \frac18\pi^2\tau^2.
  \end{align*}
  For the last equality we used the trivial evaluation
  \begin{equation}
    \LsD{n}{n-1}{\tau} = - \frac{\tau^n}{n}.
  \end{equation}
  % In terms of classical complex dilogarithms \cite{lewin2} we may write
  % \begin{align}\label{eq:gl21}
  %   \Gl{2,1}{\tau} &= \zeta(3) + \frac{\tau\tau'}2
  %   \log \left( 2\sin \frac \tau 2  \right) \nonumber\\
  %   &-\frac{\tau'}2 \Cl{2}{2\sin \frac\tau2,\frac{\tau'}2}
  %   - \Li_3 \left( 2\sin \frac \tau 2,\frac{\tau'}2 \right) \nonumber\\
  %   & + \Li_2\left(2\sin \frac \tau 2,\frac{\tau'}2\right)
  %  \log \left( 2\sin \frac \tau 2  \right)
  % \end{align}
  % where $\tau'=\pi-\tau$.  A similar but more complex form can be obtained for
  % $\Gl{3,1}{\tau}$. These can be further evaluated using the methods
  % described in \cite{lewin2} and \cite{box3} to arrive at tractable expressions
  % for specific angles.  Thence, for $\tau=\frac\pi2$,  from \eqref{eq:gl21} we
  % obtain
  % \begin{align}
  %   \Gl{2,1}{\frac\pi2} &= \frac{29}{64}\zeta(3) - \frac\pi4\G, \label{ex:pion2}\\
  %   \Gl{3,1}{\frac\pi2} &=-\frac{35}{128}\pi\zeta(3) + \Cl4{\frac\pi2}. \label{ex:pion2b}
  % \end{align}
  % Here $\G=\sum_{n \ge 0} (-1)^n/(2n+1)^2=\Cl2{\frac\pi2}$ is the \emph{Catalan constant}.
  % Likewise,
  % \begin{align*}
  %   \Gl{2,1}{\frac{2\pi}{3}} &= \frac{5}{18}\zeta(3) - \frac\pi9\Cl2{\frac\pi3}, \\
  %   \Gl{3,1}{\frac{2\pi}{3}} &= \frac{1}{24}\log^4 \sigma - \frac16\log(1-\sigma)\log^3\sigma
  %   + \frac{\pi^2}{12}\log^2 \sigma  \\
  %   &+ \left(\zeta(3)-\Li_3\left(1-\sigma \right) \right) \log\sigma \\
  %   &+ \Li_4(1-\sigma) - \Li_4(\sigma) + \Li_4\left(1-\frac{1}{\sigma}\right) + \zeta(4)
  % \end{align*}
  % where $\sigma := \frac{\sqrt{3}}{2}-1$.
  It appears that both $\Gl{2,1}{\tau}$ and $\Gl{3,1}{\tau}$ are not reducible
  for $\tau=\pi/2$ or $\tau=2\pi/3$. Here, reducible means expressible in terms
  of multi zeta values and Glaisher functions of the same argument and lower
  weight.  In the case $\tau=\pi/3$ such reductions are possible. This is
  discussed in Example \ref{eg:ls41pi3} and illustrates how much less simple
  values at $2\pi/3$ are than those at $\pi/3$. We remark, however, that
  $\Gl{2,1}{2\pi/3}$ is reducible to one-dimensional polylogarithmic terms
  \cite{logsin1}. In \cite{logsin2} explicit reductions for all weight four
  or less polylogarithms are given.
  \qede
\end{example}

\begin{remark}
   Lewin \cite[7.4.3]{lewin2} uses the special case $k=n-2$ of \eqref{eq:lsb}
   to deduce a few  small integer evaluations of the log-sine integrals
   $\LsD{n}{n-2}{\pi/3}$ in terms of classical Clausen functions.
  \qede
\end{remark}

In general, we can use \eqref{eq:lsb} recursively to express the log-sine
values $\LsD{n}{k}{\tau}$ in terms of multiple Clausen and Glaisher functions
at $\tau$.

\begin{example}
  \eqref{eq:lsb} with $n=5$ and $k=1$ produces
  \begin{align*}
    \LsD{5}{1}{\tau}
    &= -6\zeta(4,1) + 6\Cl{3,1,1}{\tau} + 6\tau\Cl{2,1,1}{\tau} \\
    & + \frac34\LsD{5}{3}{\tau} - \frac32\pi\LsD{4}{2}{\tau} + \frac34\pi^2\LsD{3}{1}{\tau}.
  \end{align*}
  Applying \eqref{eq:lsb} three more times to rewrite the remaining log-sine
  integrals produces an evaluation of $\LsD{5}{1}{\tau}$ in terms of multi zeta
  values and Clausen functions at $\tau$.
  \qede
\end{example}

\subsection[Log-sine integrals at pi/3]{Log-sine integrals at $\pi/3$}
\label{sec:lspi3}

We now apply the general results obtained in Section \ref{sec:lsx} to the
evaluation of log-sine integrals at $\tau=\pi/3$. Accordingly, we encounter
multiple polylogarithms at the basic $6$-th root of unity
$\omega:=\exp(i\pi/3)$. Their real and imaginary parts satisfy various
relations and reductions, studied in \cite{mcv}, which allow us to further
treat the resulting evaluations. In general, these polylogarithms are more
tractable than those at other values because $\overline{\omega}=\omega^2$.

\begin{example}\label{eg:ls41pi3}
  \egcaption{Values at $\frac\pi3$}
  Continuing Example \ref{eg:ls41tau} we have
  \begin{align*}
    -\LsD{4}{1}{\frac\pi3}
    &= 2\mgl{3,1} + \frac23\pi\mgl{2,1} + \frac{19}{6480}\pi^4.
  \end{align*}
  Using known reductions from \cite{mcv} we get:
  \begin{equation}\label{eq:mgl31reduce}
    \mgl{2,1} = \frac{1}{324}\pi^3,\quad
    \mgl{3,1} = -\frac{23}{19440}\pi^4,
  \end{equation}
  and so arrive at
  \begin{equation}\label{eq:queer}
    -\LsD{4}{1}{\frac\pi3} = \frac{17}{6480} \pi^4.
  \end{equation}
  Lewin explicitly mentions \eqref{eq:queer} in the preface to \cite{lewin2}
  because of its ``queer'' nature which he compares to some of Landen's curious
  18th century formulas.
  \qede
\end{example}

Many more reduction besides \eqref{eq:mgl31reduce} are known.  In particular,
the one-dimensional Glaisher and Clausen functions reduce as follows
\cite{lewin2}:
\begin{align}\label{eq:cl1d}
  \Gl{n}{2\pi x} &= \frac{2^{n-1}(-1)^{1+\lfloor n/2 \rfloor}}{n!} B_n(x)\, \pi^n , \nonumber\\
  \mcl{2n+1} &= \frac{1}{2} (1-2^{-2n}) (1-3^{-2n}) \zeta(2n+1).
\end{align}
Here, $B_n$ denotes the $n$-th \emph{Bernoulli polynomial}. Further reductions
can be derived for instance from the duality result \cite[Theorem 4.4]{mcv}.  For
low dimensions, we have built these reductions into our program, see Section
\ref{sec:reducing}.

\begin{example}\label{eg:lspi3}
  \egcaption{Values of $\Ls{n}{\pi/3}$}
  The log-sine integrals at $\pi/3$ are evaluated by our program as follows:
  \begin{align*}
    \Ls{2}{\frac\pi3} &= \mcl{2} \\
    -\Ls{3}{\frac\pi3} &= \frac{7}{108}\, \pi^3 \\
    \Ls{4}{\frac\pi3} &= \frac12\pi\,\zeta(3)+\frac 92\,\mcl{4} \\
    -\Ls{5}{\frac\pi3} &= \frac{1543}{19440}\pi^5 - 6\mgl{4,1} \\
    \Ls{6}{\frac\pi3} &= \frac{15}2\pi \,\zeta(5) + \frac{35}{36}\,\pi^3
      \zeta(3) + \frac{135}{2}\,\mcl{6} \\
    -\Ls{7}{\frac\pi3} &= \frac{74369}{326592}\pi^7 + \frac{15}{2}\,\pi
      \zeta(3)^2 - 135\,\mgl{6,1} %\\
    % \Ls{8}{\frac\pi3} &= \frac{13181}{2592}\pi^5\zeta(3) + \frac{1225}{24}\pi^3\zeta(5)
      % + \frac{319445}{864}\pi\zeta(7) \\ &+ \frac{35}{2}\pi^2\mcl{6}
      % + \frac{945}{4}\mcl{8}\\ &+ 315\mcl{6,1,1}
  \end{align*}
  As follows from the results of Section \ref{sec:lsx} each integral is a
  multivariable rational polynomial in $\pi$ as well as $\op{Cl}$, $\op{Gl}$,
  and zeta values. These evaluations confirm those given in \cite[Appendix
  A]{dk1-eps} for $\Ls{3}{\frac\pi3}$, $\Ls{4}{\frac\pi3}$, and
  $\Ls{6}{\frac\pi3}$.  Less explicitely, the evaluations of
  $\Ls{5}{\frac\pi3}$ and $\Ls{7}{\frac\pi3}$ can be recovered from similar
  results in \cite{lsjk, dk1-eps} (which in part were obtained using PSLQ; we
  refer to Section \ref{sec:reducing} for how our analysis relies on PSLQ).

  The first presumed-irreducible value that occurs is
  \begin{align}\label{eq:mgl41}
    \mgl{4,1} &= \sum_{n=1}^\infty \frac{\sum_{k=1}^{n-1}\frac{1}{k}}{n^4} \,
      \sin\left( \frac{n \pi} 3\right) \nonumber\\
    &= \frac{3341}{1632960} \pi^5 - \frac{1}{\pi}\zeta(3)^2
    - \frac{3}{4\pi} \sum_{n=1}^\infty \frac{1}{\binom{2n}{n} n^6}.
  \end{align}
  The final evaluation is described in \cite{mcv}.  Extensive computation
  suggests it is not expressible as a sum of products of one dimensional
  Glaisher and zeta values. Indeed, conjectures are made in \cite[\S5]{mcv} for
  the number of irreducibles at each depth. Related dimensional conjectures for
  polylogs are discussed in \cite{zlobin}.
  \qede
\end{example}

\subsection{Log-sine integrals at imaginary values}
\label{sec:lsimag}

The approach of Section \ref{sec:lsx} may be extended to evaluate log-sine integrals
at imaginary arguments. In more usual terminology, these are \emph{log-sinh integrals}
\begin{equation}\label{eq:lsh}
  \LshD{n}{k}{\sigma}
  := - \int_{0}^{\sigma}\theta^k\,\log^{n-1-k} \left|2\,\sinh \frac \theta 2\right| \,{\md\theta}
\end{equation}
which are related to log-sine integrals by
\begin{equation*}
  \LshD{n}{k}{\sigma} = (-i)^{k+1} \LsD{n}{k}{i\sigma}.
\end{equation*}

We may derive a result along the lines of Theorem \ref{thm:ls2li} by observing
that equation \eqref{eq:log2logsin} is replaced, when $\theta=i t$ for $t>0$, by the
simpler
\begin{equation}\label{eq:log2logsinh}
  % \log(1-\e^{i\theta}) = \log\left| 2\sin\frac\theta2 \right| + \frac{i}{2}\theta.
  \log(1-\e^{-t}) = \log\left| 2\sinh\frac{t}{2} \right| - \frac{t}{2}.
\end{equation}
This leads to:

\begin{theorem}\label{thm:lsh2li}
  For $t>0$, and nonnegative integers $n$, $k$ such that $n-k\ge2$,
  \begin{align}\label{eq:lshb}
    \zeta(&n-k,\{1\}^k) - \sum_{j=0}^{k} \frac{t^j}{j!}
    \Li_{2+k-j,\{1\}^{n-k-2}}(\e^{-t}) \nonumber\\
    % &= \frac{(-1)^{k}}{k!}\frac{(-1)^{n}}{(n-k-1)!}
    &= \frac{(-1)^{n+k}}{(n-1)!}
    \sum_{r=0}^{n-k-1} \binom{n-1}{k,r}
    \left( -\frac{1}{2} \right)^r \LshD{n}{k+r}{t}.
  \end{align}
\end{theorem}

\begin{example}
  Let $\rho:=(1+\sqrt{5})/2$ be the golden mean.
  Then, by applying Theorem \ref{thm:lsh2li} with $n=3$ and $k=1$,
  \begin{align*}
    \LshD{3}{1}{2\log\rho}
    ={}& \zeta(3) - \frac43 \log^3\rho \\
    &- \Li_3(\rho^{-2}) - 2\Li_2(\rho^{-2}) \log\rho.
  \end{align*}
  This may be further reduced, using $\Li_2(\rho^{-2}) = \frac{\pi^2}{15} -
  \log^2\rho$ and $\Li_3(\rho^{-2}) = \frac45\zeta(3) -
  \frac{2}{15}\pi^2\log\rho + \frac23\log^3\rho$, to yield the well-known
  \begin{equation*}
    \LshD{3}{1}{2\log\rho} = \frac{1}{5} \zeta(3).
  \end{equation*}
  The interest in this kind of evaluation stems from the fact that log-sinh
  integrals at $2\log\rho$ express values of alternating inverse binomial sums
  (the fact that log-sine integrals at $\pi/3$ give inverse binomial sums is
  illustrated by Example \ref{eg:lspi3} and \eqref{eq:mgl41}). In this case,
  \begin{equation*}
    \LshD{3}{1}{2\log\rho}
    = \frac12 \sum_{n=1}^\infty \frac{(-1)^{n-1}}{\binom{2n}{n} n^3}.
  \end{equation*}
  More on this relation  and generalizations can be found in each of
  \cite{williams-logsin, kv0-bin, mcv, bbg}.
  \qede
\end{example}

\section{Reducing polylogarithms}
\label{sec:reducing}

The techniques described in Sections \ref{sec:lsxpi} and \ref{sec:lsx} for
evaluating log-sine integrals in terms of multiple polylogarithms usually
produce expressions that can be considerably reduced as is illustrated in
Examples \ref{eg:ls51}, \ref{eg:ls61}, and \ref{eg:ls41pi3}. Relations between
polylogarithms have been the subject of many studies \cite{b3l,bbg} with a
special focus on (alternating) multiple zeta values
\cite{koelbig82-nielsen,hoffman-mzv,zlobin} and, to a lesser extent, Clausen
values \cite{mcv}.

There is a certain deal of choice in how to combine the various techniques
that we present in order to evaluate log-sine integrals at certain values.
The next example shows how this can be exploited to derive relations among the
various polylogarithms involved.

\begin{example}\label{eg:ls2mpi}
  For $n=5$ and $k=2$, specializing \eqref{eq:lsperiod} to $\sigma=\pi$ and $m=1$ yields
  \begin{align*}
    \LsD{5}{2}{2\pi} = 2\LsD{5}{2}{\pi} - 4\pi\LsD{4}{1}{\pi} + 4\pi^2\Ls{3}{\pi}.
  \end{align*}
  By Example \ref{eg:ls52at2pi} we know that this evaluates as $-13/45 \pi^5$.
  On the other hand, we may use the technique of Section \ref{sec:lsxpi} to reduce
  the log-sine integrals at $\pi$. This leads to
  \begin{equation*}
    -8\pi\Li_{3,1}(1) + 12\pi\Li_4(1) - \frac25 \pi^5 = -\frac{13}{45} \pi^5.
  \end{equation*}
  In simplified terms, we have derived the famous identity $\zeta(3,1) =
  \frac{\pi^4}{360}$. Similarly, the case $n=6$ and $k=2$ leads to
  $\zeta(3,1,1) = \frac32\zeta(4,1) + \frac{1}{12}\pi^2\zeta(3) - \zeta(5)$
  which further reduces to $2\zeta(5) - \frac{\pi^2}{6}\zeta(3)$.  As a final
  example, the case $n=7$ and $k=4$ produces $\zeta(5,1)=\frac{\pi^6}{1260} -
  \frac12\zeta(3)^2$.
  \qede
\end{example}

For the purpose of an implementation, we have built many reductions of multiple
polylogarithms into our program. Besides some general rules, such as
\eqref{eq:cl1d}, the program contains a table of reductions at low weight for
polylogarithms at the values $1$ and $-1$, as well as Clausen and Glaisher
functions at the values $\pi/2$, $\pi/2$, and $2\pi/3$.  These correspond to
the polylogarithms that occur in the evaluation of the log-sine integrals at
the special values $\pi/3$, $\pi/2$, $2\pi/3$, $\pi$ which are of particular
importance for applications as mentioned in the introduction. This table of
reductions has been compiled using the integer relation finding algorithm PSLQ
\cite{bbg}. Its use is thus of heuristic nature (as opposed to the rest of the
program which is working symbolically from the analytic results in this paper)
and is therefore made optional.

% to avoid an orphan
\pagebreak

\section{The program}
\label{sec:program}

\subsection{Basic usage}

% \lstset{language=Mathematica,basicstyle=\ttfamily,belowskip=-5pt}
\lstset{language=Mathematica,basicstyle=\ttfamily}

As promised, we implemented\footnote{The packages are freely available for
download from\\\url{http://arminstraub.com/pub/log-sine-integrals}}
the presented results for evaluating log-sine integrals for use in the computer
algebra systems \emph{Mathematica} and SAGE.  The basic usage is very simple
and illustrated in the next example for \emph{Mathematica}\footnote{The
interface in the case of SAGE is similar but may change slightly, especially as
we hope to integrate our package into the core of SAGE.}.

\begin{example}\label{eg:mathematica}
  Consider the log-sine integral $\LsD{5}{2}{2\pi}$. The following
  self-explanatory code evaluates it in terms of polylogarithms:
  \begin{lstlisting}
  LsToLi[Ls[5,2,2Pi]]
  \end{lstlisting}
  This produces the output $-13/45\pi^5$ as in Example \ref{eg:ls52at2pi}.
  As a second example,
  \begin{lstlisting}
  -LsToLi[Ls[5,0,Pi/3]]
  \end{lstlisting}
  results in the output
  \begin{lstlisting}
  1543/19440*Pi^5 - 6*Gl[{4,1},Pi/3]
  \end{lstlisting}
  which agrees with the evaluation in Example \ref{eg:lspi3}.
  Finally,
  \begin{lstlisting}
  LsToLi[Ls[5,1,Pi]]
  \end{lstlisting}
  produces
  \begin{lstlisting}
  6*Li[{3,1,1},-1] + (Pi^2*Zeta[3])/4
  - (105*Zeta[5])/32
  \end{lstlisting}
  as in Example \ref{eg:ls51}.
  \qede
\end{example}

\begin{example}
  Computing
  \begin{lstlisting}
  LsToLi[Ls[6,3,Pi/3]-2*Ls[6,1,Pi/3]]
  \end{lstlisting}
  yields the value $\frac{313}{204120}\pi^6$ and thus automatically proves a
  result of Zucker \cite{zucker-logsin}.  A family of relations between
  log-sine integrals at $\pi/3$ generalizing the above has been established in
  \cite{williams-logsin}.
  \qede
\end{example}

\subsection{Implementation}

The conversion from log-sine integrals to polylogarithmic values demonstrated
in Example \ref{eg:mathematica} roughly proceeds as follows:
\begin{itemize}\addtolength{\itemsep}{-0.5\baselineskip}
  \item First, the evaluation of $\LsD{n}{k}{\sigma}$ is reduced to the cases
    of $0\le\sigma\le\pi$ and $\sigma=2m\pi$ as described in Section
    \ref{sec:lsreduce}.
  \item The cases $\sigma=2m\pi$ are treated as in Section \ref{sec:ls2pi} and
    result in multiple zeta values.
  \item The other cases $\sigma$ result in polylogarithmic values at
    $e^{i\sigma}$ and are obtained using the results of Sections
    \ref{sec:lsxpi} and \ref{sec:lsx}.
  \item Finally, especially in the physically relevant cases, various
    reductions of the resulting polylogarithms are performed as outlined in
    Section \ref{sec:reducing}.
\end{itemize}

\subsection{Numerical usage}

The program is also useful for numerical computations provided that it is
coupled with efficient methods for evaluating polylogarithms to high precision.
It complements for instance the C++ library \texttt{lsjk} ``for
arbitrary-precision numeric evaluation of the generalized log-sine functions''
described in \cite{lsjk}.

\begin{example}
  We evaluate
  \begin{align*}
    \LsD{5}{2}{\frac{2\pi}{3}}
    ={}& 4\Gl{4,1}{\frac{2\pi}{3}} - \frac83\pi \Gl{3,1}{\frac{2\pi}{3}} \\
    &- \frac89 \pi^2 \Gl{2,1}{\frac{2\pi}{3}} - \frac{8}{1215}\pi^5.
  \end{align*}
  Using specialized code\footnote{The C++ code we used is based on the fast
  H\"older transform described in \cite{b3l}, and is available on request.}
  such as \cite{polylog-num}, the right-hand side is readily evaluated to, for
  instance, two thousand digit precision in about a minute. The first 1024
  digits of the result match the evaluation given in \cite{lsjk}. However, due
  to its implementation \texttt{lsjk} currently is restricted to log-sine
  functions $\LsD{n}{k}{\theta}$ with $k\le9$.
  % We remark that, in this particular
  % example, the Glaisher functions $\op{Gl}_{2,1}$ and $\op{Gl}_{3,1}$ may be
  % further reduced as shown in Example \ref{eg:ls41tau} while
  % $\Gl{4,1}{\frac{2\pi}{3}}$ like $\Gl{4,1}{\frac{\pi}{3}}$ appears to be
  % irreducible.
  \qede
\end{example}

\paragraph{Acknowledgements}
We are grateful to Andrei Davydychev and Mikhail Kal\-my\-kov for several valuable
comments on an earlier version of this paper and for pointing us to relevant
publications. We also thank the reviewers for their thorough reading and
helpful suggestions.

\bibliography{logsin-refs}
\bibliographystyle{abbrveprint}

\end{document}